\documentclass{amsart}
 	\newsymbol\varnothing 203F
	\newsymbol\subsetneqq 2324
 	\newsymbol\diagdown 231F
	\newsymbol\digamma 207A

	\let\<=\langle
	\let\>=\rangle
	\let\ssen=\subsetneqq
	\let\wtil=\widetilde
	\let\what=\widehat
	\let\\=\backslash
	\let\sse=\subseteq
	
	\let\vtheta=\vartheta

	\def\newmatrix#1{\null\,\vcenter{
			\baselineskip=8pt\mathsurround=-0pt\ialign{
			\hfil ${##}$
			\hfil &&
			\hfil ${##}$
			\hfil \crcr
			\mathstrut \crcr
			\noalign{\kern-\baselineskip}#1 \crcr
			\mathstrut \crcr
			\noalign{\kern-\baselineskip} \crcr }}\!}

	\def\CC{{\mathbb C\kern.5pt}}
	\def\FF{{\mathbb F\kern.5pt}}
	\def\NN{{\mathbb N\kern.5pt}}
	\def\RR{{\mathbb R\kern.5pt}}
	\def\SSs{{\mathbb S\kern.5pt}}
	
	\def\F{{\mathcal F}}
	
	\def\Le{{\mathcal L}}
	\def\M{{\mathcal M}}
	\def\N{{\mathcal N}}
 	\def\Se{{\mathcal S}}
	\def\SSe{{\mbox{{$\Se\kern-6pt/$}}\kern.5pt}}
	\def\sSSe{{\mbox{{$\scriptstyle\Se\kern-4.5pt/$}}}}
	\def\R{{\mathcal R}}
	\def\T{{\mathcal T}}
 	\def\U{{\mathcal U\kern.5pt}}
 	\def\V{{\kern-.5pt\mathcal V}}
	\def\W{{\mathcal W}}
	\def\X{{\mathcal X}}
	\def\Y{{\mathcal Y}}
	\def\Z{{\mathcal Z}}

	\def\noi{\noindent}
	\def\0{{\{0\}}}

	\def\x{{\times}}
	
 	\def\Ps{{\vbox{\hbox{$\wp\kern.3pt$}\vskip.2pt}}}
 	
	\def\vchi{{^{_{\textstyle\kern1pt\chi\kern-1pt}}}}

	\def\emap{\hbox to25pt{\rightarrowfill}}
	
	\def\nmap{\Big\uparrow}

	\def\diagdownBOX{\hbox{$\diagdown$}}
	\def\searrowBOX{\hbox{\hglue6.5pt$\searrow$}}

	\def\semap{\vbox{\offinterlineskip
			\diagdownBOX\vglue-1pt\searrowBOX\vglue-6pt}}

	\def\span{{\kern.5pt{\rm span}\kern2pt}}
	\def\QED{{\hfill\hfill\qed}}

	\newtheorem{theorem}{Theorem}[section]
	\newtheorem{corollary}{Corollary}[section]
	\newtheorem{proposition}{Proposition}[section]
	
	\theoremstyle{definition}
	\newtheorem{remark}{Remark}[section]
	
	\newtheorem{definition}{Definition}[section]

\begin{document}

\vglue-75pt\noi
\hfill{\it Bulletin of the Malaysian Mathematical Sciences Society}\/,
{\bf 44} (2021) 2335--2355

\vglue28pt
\title
[Tensor Products Revisited: Axiomatic Approach]
{Algebraic Tensor Products Revisited: Axiomatic Approach}
\author{C.S. Kubrusly}
\address{Mathematics Institute, Federal University of Rio de Janeiro, Brazil}
\email{carloskubrusly@gmail.com}
\renewcommand{\keywordsname}{Keywords}
\keywords{Bilinear maps, quotient spaces, tensor product}
\subjclass{46M05, 47A80, 15A63}
\date{July 10, 2020; revised, November 24, 2020}

\begin{abstract}
This is an expository paper on tensor products where the standard approaches
for constructing concrete instances of algebraic tensor products of linear
spaces, via quotient spaces or via linear maps of bilinear maps, are reviewed
by reducing them to different but isomorphic interpretations of an abstract
notion, viz., the universal property, which is based on a pair of axioms.
\end{abstract}

\maketitle
\kern-20pt

\section{Introduction}

The purpose of this paper is to offer a brief and unified review with an
expository flavor on the common realizations of algebraic tensor products
(either via quotient space or via linear maps of bilinear maps) by reversing
the presentation order$.$ In other words, this exposition focuses on the
approach where the so-called universal property is taken as an axiomatic
starting point, rather than as a theorem for a specific realization$.$ This
leads to the abstract notion of algebraic tensor products of linear spaces
(the pre-crossnorm stage), where the concrete standard forms are shown to
be interpretations of the axiomatic formulation.

\vskip4pt
The origin of a systematic presentation of tensor products in book form dates
back to Schatten's 1950 monograph \cite{Sch}, where the notion of {\it direct
product}\/ of linear spaces was given in terms of {\it formal products}\/
(which match what its now called {\it tensor product space}\/ and {\it single
tensors}\/) and their symbols in the form of finite sum of formal products$.$
This followed a Kronecker-product-like notion on finite-dimensional spaces
given in Weyl's 1931 book \cite[Chapter V]{Wey}$.$ Grothendieck's fundamental
work in the 1950's has been unified and updated by Diestel, Fourie and Swart
in 2008 \cite{DFS}$.$ See also Pisier's 2012 exposition \cite{Pis}$.$ In
Grothendieck's pioneering work, the notion of tensor product space was
essentially given in terms of the dual of the linear space of bilinear forms
(or, more generally, the linear space of linear maps of bilinear maps)$.$ The
same way of defining tensor product also appears in \hbox{Halmos's} 1958 book
on finite-dimensional vector spaces \cite[Section 24]{Hal} (although the
formal products variant is also mentioned as a possible alternative)$.$
Another representative along this line (dual of the linear space of bilinear
forms) is Ryan's 2002 book on tensor products of Banach spaces \cite{Rya}.

\vskip4pt
On the other hand, a different but still usual approach for defining tensor
product relies on quotient spaces of free linear spaces
$\kern-.5pt$(equivalent to the linear space of formal linear
combinations)$\kern-.5pt$ of Cartesian products of linear spaces$.\!$ This has
been \hbox{sometimes} referred to as {\it algebraic tensor product}\/
(although the previous approach is equally algebraic)$.$ See e.g.,
\cite[pp.22--25]{BP}, \cite[Section 3.4]{Wei} and \cite[Chapter 14]{Rom} for a
linear space version, and \cite[Section IX.8]{MB} and
\cite[Section XVI.1]{Lan}) for a module version.

\vskip4pt
The present paper is organized as follows$.$ Notation, terminology and
auxiliary results are brought together in Section 2$.$ This is split into
three parts$:\kern-1pt$ formal linear combination, quotient space, and
bilinear maps$.$ Tensor products are axiomatically defined in Section 3, and
common properties shared by the concrete formulations are obtained from such
an abstract formulation$.$ The usual realizations, viz., the quotient space
and linear maps of bilinear maps, are individually inferred from
Definition 3.1 and are considered in Sections 4 and 5.

\section{Notation, Terminology, and Auxiliary Results}

Let $\X$ and $\Y$ be linear spaces over the same field $\FF$, and let
$\Le[{\X,\Y}]$ denote the linear space over $\FF$ of all linear transformations
from $\X$ to $\Y.$ For ${\X=\Y}$ write $\Le[\X]=\Le[\X,\X]$, which is the
algebra of all linear transformations of $\X$ into itself$.$ The kernel
${\N(L)=L^{-1}(\0)}$ and range ${\R(L)=L(\X)}$ of ${L\in\Le[\X,\Y]}$ are linear
manifolds of $\X$ and $\Y$ respectively$.$ A linear transformation
${L\in\Le[\X,\Y]}$ is injective if $\N(L)=\0$ and surjective if $\R(L)=\Y.$ By
an {\it isomorphism}\/ (or an {\it algebraic isomorphism}\/, or a {\it
linear-space isomorphism}\/) we mean an invertible (i.e., injective and
surjective) linear transformation$.$ Two linear spaces $\X$ and $\Y$ are
{\it isomorphic}\/ (notation ${\X\cong\Y}$) if there is an isomorphism
between them$.$ For the particular case of ${\Y=\FF}$ the elements of
$\Le[{\X,\FF}]$ are referred to as {\it linear functionals}\/ or {\it linear
forms}\/, and the linear space $\Le[{\X,\FF}]$ of all linear functionals on
$\X$ is referred to as the {\it algebraic dual}\/ of $\X$, denoted by
$\X^\sharp$ (i.e., ${\X^\sharp=\Le[{\X,\FF}]}).$ For an arbitrary linear
transformation ${L\in\Le[\X,\Y]}$ consider the linear transformation
${L^\sharp\in\!\Le[\Y^\sharp\!,\X^\sharp]}$ defined by
${L^\sharp g=g L\in\X^\sharp}$ for every ${g\in\!\Y^\sharp}\!$
(i.e., ${(L^\sharp g)(x)=g(Lx)\in\FF}$ for every ${g\in\!\Y^\sharp}$ and
every ${x\in\!\X}$ --- we use both notations ${gL}$ or ${g\circ L}$ for
composition)$.$ This ${L^\sharp\in\!\Le[\Y^\sharp\!,\X^\sharp]}$ is the
{\it algebraic adjoint}\/ of ${L\in\Le[\X,\Y]}$.

\vskip4pt
The next subsections summarize not only notation and terminology, but also
auxiliary results that will be required in Sections 4 and 5.

\subsection{Formal Linear Combination}
Let $S$ be an arbitrary nonempty set and let $\FF$ be a field$.$ Consider the
linear space $\FF^S\!$ of all scalar-valued functions ${f\!:S\to\FF}$ on $S.$
Let $\#$ stand for cardinality and consider the set
$$
\SSe=\big\{f\in\FF^S\!:f(S\\A)=0
\;\hbox{for some}\;A\sse S\;\hbox{with}\;\#A<\infty\big\}
$$
of all functions ${f\!:S\to\FF}$ which vanish everywhere on the complement
of some finite subset $A$ of $S$ (which depends on $f).$ This $\SSe$ is a
linear manifold of $\FF^S\!$, and so is itself a linear space over $\FF.$ For
each ${s\in S}$ take the characteristic function
${e_s=\vchi_{\{s\}}\!:S\to\FF}$ of the singleton ${\{s\}\sse S}.$ As is
readily verified the set
$$
\hbox{{\small$\SSs$} $=\{e_s\}_{s\in S}$ is a Hamel basis for the linear
space $\SSe$}.
$$
Thus an arbitrary vector ${f\in\SSe}$, being a scalar-valued function taking
nonzero values only over a finite subset $\{s_i\}_{i=1}^n$ of $S$, has a
unique expansion with ${\alpha_i\in\FF}$:
$$
f={\sum}_{i=1}^n\alpha_ie_{s_i}\,\in\,\SSe.
$$
The linear space $\SSe$ is called the {\it free linear space generated by}\/
$S.$ Since $\#\{e_s\}_{s\in S}=$ $\#S$ (i.e., each element $s$ from the set
$S$ is in a one-to-one correspondence with each function $e_s$ from the Hamel
basis \hbox{\small$\SSs$} for the function space $\SSe$), then
${\dim\SSe\kern-1pt=\kern-1pt\#\hbox{\small$\SSs$}\kern-1pt=\kern-1pt\#S}.$
This sets a natural identification $\approx$ such that ${s\approx e_s}$ and so
$S\approx\hbox{\small$\SSs$}$, which in turn leads to a natural identification
for an arbitrary linear combination in $\SSe$,
$$
{\sum}_{i=1}^n\alpha_ie_{s_i}\approx{\sum}_{i=1}^n\alpha_is_i,
$$
where ${\sum}_{i=1}^n\alpha_is_i$ is referred to as a {\it formal linear
combination of points}\/ ${s_i\in S}$ (although addition or scalar
multiplication is not directly defined on the set $S$), the collection of
which is the {\it linear space of formal linear combinations from}\/ $S.$ So
any function $f$ in the linear space $\SSe$ is identified with a formal linear
combination of points in $S$, and the set $S$ that generates the free linear
space $\SSe$ is identified with the Hamel basis \hbox{\small$\SSs$} for $\SSe.$
In this sense the set $S$ may be regarded as a subset of $\SSe$, and a
function $f$ in $\SSe$ may be regarded as a formal linear combination$.$ Thus
write
$$
f={\sum}_{i=1}^n\alpha_is_i
\;\;\;\hbox{for}\;\;\;
{\sum}_{i=1}^n\alpha_is_i\approx f\in\SSe
\qquad\hbox{and}\qquad
S\subset\SSe
\;\;\;\hbox{for}\;\;\;
S\approx\hbox{\small$\SSs$}\subset\SSe.
$$

\subsection{Quotient Space}
Let $\M$ be a linear manifold of a linear space $\X$ over a field $\FF$, let
$[x]=x+\M$ denote the coset of ${x\in\X}$ modulo $\M$, and let $\X/\M$ stand
for the quotient space of $\X$ modulo $\M$, which is the linear space over
$\FF$ of all cosets $[x]$ modulo $\M$ for every ${x\in\X}.$ Consider the
{\it quotient map}\/ (or the {\it natural quotient map}\/)
${\pi\!:\X\to\X/\M}$ of the linear space $\X$ onto the linear space $\X/\M$,
defined by
$$
\pi(x)=[x]=x+\M
\,\;\;\hbox{for every}\;\;
x\in\X,
$$
which is a surjective linear transformation according to the usual definition
of addition and scalar multiplication in $\X/\M$.

\begin{proposition}
{\rm Universal Property.}
Let\/ $\X$ and\/ $\Z$ be linear spaces over the same field, let\/ $\M$
be a linear manifold of\/ $\X$, consider the quotient space\/ ${\X/\M}$, and
take the natural quotient map\/ ${\pi\! :\X\to\X/\M}.$ If\/ ${L\in\Le[\X,\Z]}$
and if\/ ${\M\sse\N(L)}$, then there exists a unique\/
${\what L\in\Le[\X/\M,\Z]}$ such that
$$
L=\what L\circ\pi.
$$
\vskip-4pt\noi
In other words, the diagram
\vskip-2pt\noi
$$
\newmatrix{
\X & \kern2pt\buildrel L\over\emap & \kern-1pt\Z                      \cr
   &                               &                                  \cr
   & \kern-3pt_\pi\kern-3pt\semap  & \kern4pt\nmap\scriptstyle\what L \cr
   &                               & \phantom{;}                      \cr
   &                               & \kern-2pt\X/\M                   \cr}
$$
commutes, which means the quotient map\/ $\pi$ factors the linear
transformation $L$ through\/ $\X/\M.$ Moreover, in this case\/
${\N(\what L)=\N(L)/\M}$ and\/ ${\R(\what L)=\R(L)}$.
\end{proposition}

\begin{proof}
See, e.g., \cite[Theorem 3.4, 3.5]{Rom}$.$ (For a module version see, e.g.,
\cite[Theorem V.4.7]{MB}; for a normed space version see, e.g.,
\cite[Theorem 1.7.13]{Meg}).
\end{proof}

\begin{remark}
Let $\M$ be linear a manifold of a linear space $\X.$ A set
$\{\kern1pt[\hbox{e}_\delta]\}_{\delta\in\Delta}$ of cosets modulo $\M$ (with
each $\hbox{e}_\delta$ in $\X$) is a Hamel basis for the linear space $\X/\M$
if and only if $\{\hbox{e}_\delta\}_{\delta\in\Delta}$ is a Hamel basis for
some algebraic complement of $\M$ (see, e.g., \cite[Remark A.1(b)]{ST2})$.$
Hence every Hamel basis $\{e_\gamma\}_{\gamma\in\Gamma}$ for $\X$ is such that
$\{\kern1pt[\hbox{e}_\gamma]\}_{\gamma\in\Gamma}$ includes a Hamel basis for
$\X/\M.$ Since the quotient map ${\pi\!:\X\to\X/\M}$ is surjective, the image
$\pi(E_\X)$ of an arbitrary Hamel basis $E_\X$ for $\X$ includes a Hamel basis
$E_{\X/\M}$ for $\X/\M$ (i.e., $E_{\X/\M}\sse \pi(E_\X)).$ Therefore
$$
\span\pi(E_\X)=\X/\M
\quad
\hbox{for every Hamel basis $E_\X$ for $\X$},
$$
and so every element $[x]$ of $\X/\M$ can be written as a (finite) linear
combination of images under $\pi$ of elements of an arbitrary Hamel basis
for $\X$.
\end{remark}

\subsection{Bilinear Maps}
Let $\X$, $\Y$, $\Z$ be nonzero linear spaces over a field $\FF.$ Take the
Cartesian product ${\X\x\Y}$ (no algebraic structure imposed on ${\X\x\Y}$
besides the fact that both $\X$ and $\Y$ are linear spaces)$.$ A
{\it bilinear map}\/ ${\phi\!:\X\x\Y\to\Z}$ is a function from the Cartesian
product ${\X\x\Y}$ of linear spaces to a linear space $\Z$ whose sections
are linear transformations. Precisely, let
$\phi^y\!=\phi(\cdot,y)={\phi|_{\X\x\{y\}}\!:\X\!\to\Z}$ be the $y$-section of
$\phi$ and let $\phi_x\!=\phi(x,\cdot)={\phi|_{\{x\}\x\Y}\!:\Y\to\Z}$ be the
$x$-section of $\phi.$ These functions $\phi^y$ and $\phi_x$ are linear
transformations: $\phi^y={\phi(\cdot,y)\in\Le[\X,\Z]}$ for each $y$ in $\Y$ and
$\phi_x={\phi(x,\cdot)\in\Le[\Y,\Z]}$ for each $x$ in $\X.$ Let $\Z^S\!$ denote
the linear space over the same field $\FF$ of all $\Z$-valued functions on a
set $S$, and let
$$
b[\X\x\Y,\Z]
=\big\{\phi\in\kern-1pt\Z^{\X\x\Y}\!\!:\phi\;\hbox{is bilinear}\big\}
$$
stand for the collection of all $\Z$-valued bilinear maps on ${\X\x\Y}.$ The
particular case of $\Z=\FF$ yields a {\it bilinear functional}\/
${\phi\!:\X\x\Y\to\FF}$, also referred to as a {\it bilinear form}\/$.$
Bilinearity of elements $\phi$ in $b[\X\x\Y,\Z]$ ensures $b[\X\x\Y,\Z]$ is a
linear manifold of the linear space $\Z^{\X\x\Y}\!$, thus a linear space over
$\FF$ itself$.$ Let $R(\phi)=\phi(\X\x\Y)$ denote the range of
${\phi\in b[\X\x\Y,\Z]}$, which in general is not a linear manifold of $\Z$.

\vskip4pt
$\!$As a composition of linear transformations is a linear transformation, a
composi\-tion of a bilinear map with a linear transformation is a bilinear
map$.$ Also, a restriction of a bilinear map to a Cartesian product of
linear manifolds is again a bilinear map (as a consequence of
the definitions of linear manifold and of bilinear map)$.$

\begin{proposition}
Let\/ ${\X,\Y,\Z}$ be linear spaces over the same field\/ $\FF$,
and let\/ $\M$ and\/ $\N$ be nonzero linear manifolds of\/ $\X$ and\/ $\Y.$
If\/ ${\phi\!:\M\x\N\to\Z}$ is a bilinear map, then there exists a bilinear
extension\/ ${\what\phi\!:\X\x\Y\to\Z}$ of\/ $\phi$ defined on the whole
Cartesian product\/ ${\X\x\Y}$ of the linear spaces\/ $\X$ and\/ $\Y$.
\end{proposition}

\vskip0pt\noi
{\it Sketch of Proof}\/.
Every linear transformation on a linear manifold of a linear space has a
linear extension to the whole space, whose proof is an application of
Zorn's Lemma (see, e.g., \cite[Theorem 2.9]{EOT})$.$ If a bilinear map
is the product of two (alge\-bra-valued) linear maps, then the proof is an
application of the linear case$.$ A proof for the general bilinear case
follows an argument similar to the linear case. \QED

\vskip6pt
Bilinear maps can also be extended by factoring them by the natural bilinear
map through a tensor product space (see, e.g., \cite[p.101]{Hay})$.$
Indeed, Proposition 3.3 will say that ${b[\X\x\Y,\Z]\cong\Le[\X\otimes\Y,\Z]}$,
where ${\X\otimes\Y}$ stands for tensor product, and this ensures bilinear
extension out of linear extension. It is, however, advisable to have the
above extension result independently of the notion of tensor product.

\section{Tensor Product of Linear Spaces: Axiomatic Theory}

\begin{definition}
Let $\X$ and $\Y$ be nonzero linear spaces over the same field $\FF.$ A
{\it tensor product of\/ $\X$ and}\/ $\Y\,$ is a pair ${(\T,\theta)}$
consisting of a linear space $\T$ over $\FF$ and a bilinear map
${\theta\!:\X\x\Y\to\T}$ fulfilling the following two axioms.
\vskip4pt\noi
(a)
The bilinear map ${\theta\in b[\X\x\Y,\T]}$ is such that its range
$R(\theta)$ spans $\T$.
\vskip4pt\noi
(b)
If ${\phi\in b[\X\x\Y,\Z]}$ is an arbitrary bilinear map into a linear space
$\Z$ over $\FF$, then there is a linear transformation
${\Phi\in\Le[\T,\Z]}$ for which
$$
\phi=\Phi\circ\theta.
$$
\vskip-4pt\noi
That is, the diagram
\vskip-2pt\noi
$$
\newmatrix{
\X\x\Y & \kern2pt\buildrel\phi\over\emap & \kern-1pt\Z                   \cr
       &                                 &                               \cr
       & \kern-3pt_\theta\kern-3pt\semap & \kern4pt\nmap\scriptstyle\Phi \cr
       &                                 & \phantom{;}                   \cr
       &                                 & \kern-2pt\T                   \cr}
$$
commutes, and so $\theta$ factors every bilinear map through $\T.$ The linear
space $\T$ is referred to as a {\it tensor product space of $\X$ and}\/ $\Y$
associated with $\theta$, and $\theta$ is referred to as the\/ {\it natural
bilinear map}\/ (or simply the {\it natural map}\/) associated with $\T$.
\end{definition}

There are different interpretations of tensor products. For instance, the
quotient space and the linear maps of bilinear maps formulations are examples
of common procedures for building tensor products$.$ These will be shown to be
isomorphic, and will be exhibited in the next two sections$.$ So the existence
of tensor products will be postponed until then$.$ Definition 3.1 is our
starting point for providing these interpretations$.$ A similar start has been
considered, for instance, in \cite[Section XVI.1]{Lan} and
\cite[Chapter 14]{Rom} for the quotient space formulation, in
\cite[Section 1.4]{Yok} and \cite[Chapter1]{Gre} for both formulations, and in
\cite[Section 1.6]{Jar} and \cite[Section 2.2]{DF} for the linear maps of
bilinear maps formulation aiming at tensor norms.

\vskip4pt
The value $\theta(x,y)$ of the natural bilinear map
${\theta\!:\X\x\Y\to\T=\span\theta(\X\x\Y)}$ associated with $\T$ at a pair
${(x,y)}$ in ${\X\x\Y}$ is denoted by ${x\otimes y}$\/,
$$
x\otimes y=\theta(x,y)\,\in\,\T
\quad\;\hbox{for every}\;\quad
(x,y)\in\X\x\Y,
$$
and ${x\otimes y}$ is called a {\it single tensor}\/ (or a {\it decomposable
element}\/, or a {\it single tensor product of $x$ and}\/ $y$) in the tensor
product space $\T.$ Take ${(x,y)\in\X\x\Y}$ and ${\alpha\in\FF}$ arbitrary$.$
Since $\theta$ is bilinear, then
${\alpha(x\otimes y)}={(\alpha\kern.5ptx)\otimes y}
={x\otimes(\alpha\kern.5pty)}$
for any ${\alpha\in\FF}.$ So a multiple of a single tensor is again a single
tensor, and the representation of a nonzero single tensor is not unique$.$
Proofs for the next two propositions are straightforward form Definition 3.1,
thus omitted.

\begin{proposition}
An element of a tensor product space is represented as a finite sum
of single tensors\/ $($and such a representation is not unique$)$\/:
\vskip0pt\noi
$$
\digamma\in\T
\quad\iff\quad
\digamma={\sum}_i x_i\otimes y_i
\quad(\hbox{{\it a finite sum}}).
$$
\end{proposition}

\begin{proposition}
If\/ $\T$ is a tensor product space, then the linear transformation\/
${\Phi\in\Le[\T,\Z]}$ associated with each ${\phi\in b[\X\x\Y,\Z]}$ as in
Definition 3.1 is unique\/.
\end{proposition}

The natural bilinear map $\theta$ associated with a tensor product space $\T$
is unique and, conversely, $\T$ associated with $\theta$ is unique$.$ This is
shown in the next theorem and its corollary$.\!$ From now on all linear
spaces anywhere are over the same field $\FF\kern-1pt$.

\begin{theorem}
Let\/ $\X$ and\/ $\Y$ be linear spaces$.$ Let\/ ${(\T,\theta)}$ and\/
${(\T',\theta')}$ be tensor products of\/ $\X$ and\/ $\Y.$ Then there is a
unique isomorphism\/ ${\Theta\in\Le[\T',\T]}$ such that\/
${\Theta\circ\theta'=\theta}.$ That is, such that the following diagram
commutes\/:
\vskip1pt\noi
$$
\newmatrix{
\X\x\Y & \kern2pt\buildrel\theta\over\emap & \kern-1pt\T                 \cr
   &                                    &                                \cr
   & \kern-3pt_{\theta'}\kern-3pt\semap &\kern4pt\nmap\scriptstyle\Theta \cr
   &                                    & \phantom{;}                    \cr
   &                                    & \kern-2pt\T'.                  \cr}
$$
\end{theorem}

\begin{proof}
Let ${(\T,\theta)}$ and ${(\T'\!,\theta')}$ be tensor products of $\X$ and
$\Y.$ For any bilinear map ${\phi\!:\X\x\Y\to\Z}$ into a linear space $\Z$
there are linear transformations ${\Phi\!:\T\!\to\Z}$ and
${\Phi'\!:\T'\!\to\Z}$ such that
${\phi=\Phi\circ\theta}={\Phi'\!\circ\theta'}\!$ (Definition 3.1), which
means the diagrams
\vskip0pt\noi
$$
\newmatrix{
\X\x\Y & \kern2pt\buildrel\phi\over\emap & \kern-1pt\Z                   \cr
       &                                 &                               \cr
       & \kern-3pt_\theta\kern-3pt\semap & \kern4pt\nmap\scriptstyle\Phi \cr
       &                                 & \phantom{;}                   \cr
       &                                 & \kern-2pt\T                   \cr}
\qquad\hbox{and}\qquad
\newmatrix{
\X\x\Y & \kern2pt\buildrel\phi\over\emap & \kern-1pt\Z                   \cr
 &                                    &                                  \cr
 & \kern-3pt_{\theta'}\kern-3pt\semap & \kern4pt\nmap\scriptstyle{\Phi'} \cr
 &                                    & \phantom{;}                      \cr
 &                                    & \kern-2pt\T'                     \cr}
$$
commute$.$ Since ${\theta'\!:\X\x\Y\to\T'}\!$ and ${\theta\!:\X\x\Y\to\T}$
are bilinear maps, then there are linear transformations
${\Theta'\!:\T\!\to\T'}\!$ and ${\Theta\!:\T'\!\to\T}\!$ such that
${\theta'\!=\Theta'\!\circ\theta}$ and ${\theta=\Theta\!\circ\theta'}\!$
(Definition 3.1 again), which means the diagrams
$$
\newmatrix{
\X\x\Y & \kern2pt\buildrel{\theta'}\over\emap & \kern-1pt\T'             \cr
  &                                 &                                    \cr
  & \kern-3pt_\theta\kern-3pt\semap & \kern4pt\nmap\scriptstyle{\Theta'} \cr
  &                                 & \phantom{;}                        \cr
  &                                 & \kern-2pt\T                        \cr}
\qquad\hbox{and}\qquad
\newmatrix{
\X\x\Y & \kern2pt\buildrel\theta\over\emap & \kern-1pt\T                 \cr
  &                                    &                                 \cr
  & \kern-3pt_{\theta'}\kern-3pt\semap & \kern4pt\nmap\scriptstyle\Theta \cr
  &                                    & \phantom{;}                     \cr
  &                                    & \kern-2pt\T'                    \cr}
$$
\vskip-2pt\noi
commute$.$ Therefore
$$
\theta'\!=(\Theta'\!\circ\Theta)\circ\theta'
\quad\;\hbox{and}\;\quad
\theta=(\Theta\circ\Theta')\circ\theta.
$$
Let $I_S$ denote the identity function on an arbitrary set $S.$ By the above
equations,
$$
\Theta'\!\circ\Theta|_{R(\theta')}=I_{R(\theta')}\!:R(\theta')\to R(\theta')
\quad\;\hbox{and}\;\quad
\Theta\circ\Theta'|_{R(\theta)}=I_{R(\theta)}\!:R(\theta)\to R(\theta).
$$
\vskip2pt\noi
Since $\Theta$ and $\Theta'$ are linear transformations (and so are their
compositions), and since $\span R(\theta)=\T$ and $\span R(\theta')={\T'}$
by Definition 3.1, then we get
$$
\Theta'\!\circ\Theta=\Theta'\!\circ\Theta|_{\span\!R(\theta')}
=I_{\span\!R(\theta')}=I_{\T'},
$$
$$
\Theta\circ\Theta'=\Theta\circ\Theta'|_{\span\!R(\theta)}
=I_{\span\!R(\theta)}=I_\T.
$$
Hence $\Theta$ and $\Theta'\!$ are the inverse of each other, and are unique
by Proposition 3.2.
\end{proof}

\begin{corollary}
A tensor product of linear spaces is unique up to an isomor\-phism in the
following sense$.$ If\/ ${(\T,\theta)}$ and\/ ${(\T'\!,\theta')}$ are tensor
products of the same pair of linear spaces, then\/
${(\T,\theta)}={(\Theta\T'\!,\Theta\theta')}$ for an isomorphism\/ $\Theta$
in ${\Le[\T',\T]}.$ In particular, two tensor product spaces of the same pair
of linear spaces coincide if and only if the natural bilinear maps coincide\/.
\end{corollary}

\begin{proof}
This is an immediate consequence of Theorem 3.1.
\end{proof}

A tensor product for a given pair of linear spaces is unique up to an
isomorphism by Corollary 3.1$.$ Then for a given pair ${(\X,\Y)}$ of
linear space it is common to write
$$
\T=\X\otimes\Y
$$
for {\it the}\/ tensor product space, and ${(\X\otimes\Y,\theta)}$ for
{\it the}\/ tensor product, of $\X$ and $\Y$.

\begin{proposition}
{\it Take an arbitrary triple\/ ${(\X,\Y,\Z)}$ of linear spaces$.$ The linear
spaces\/ $b[{\X\x\Y,\Z}]$ and\/ $\Le[{\T,\Z]}$ are isomorphic}\/:
$$
b[{\X\x\Y,\Z}]\cong\Le[{\X\otimes\Y,\Z}].
$$
\end{proposition}

\begin{proof}
Take any ${\Phi\in\Le[\T,\Z]}.$ The composition
${\Phi\circ\theta\!:\X\x\Y\to\Z}$ lies in $b[{\X\x\Y,\Z}]$ since $\theta$ is
bilinear and $\Phi$ is linear$.$ Let ${C_\theta\!:\Le[\T,\Z]\to b[\X\x\Y,\Z]}$
be defined by
$$
C_\theta(\Phi)=\Phi\circ\theta\,\in\,b[\X\x\Y,\Z]
\quad\;\hbox{for every}\,\quad
\Phi\in\Le[\T,\Z].
$$
For every $\phi$ in $b[{\X\x\Y,\Z}]$ there is one and only one $\Phi$ in
$\Le[{\T,\Z}]$ for which ${\phi=\Phi\circ\theta}$ according to
Proposition 3.2$.$ Then $C_\theta$ is injective and surjective, and so
invertible$.$ Since it is readily verified that $C_\theta$ is linear, then
$C_\theta$ is an isomorphism.
\end{proof}

Thus a crucial property of a tensor product ${(\T,\theta)}$ is to linearize
bilinear maps (via factorization by $\theta$ through $\T$) in the sense of
in Proposition 3.3$.$ In particular,
$$
(\X\otimes\Y)^\sharp\cong b[\X\x\Y,\FF].
$$

\begin{theorem}
Let\/ ${\T=\X\otimes\Y}$ be a tensor product space of\/ $\X$ and\/ $\Y,$
let\/ $E$ and\/ $D$ be nonempty subsets of\/ $\X$ and $\Y$ respectively,
and set
$$
\top_{\!E,D}
=\big\{x\otimes y\in\X\otimes\Y\!:\,x\in E\;\hbox{and}\;\,y\in D\big\}.
$$
\begin{description}
\item {$\kern-9pt$\rm(a)}
If\/ $\span E=\X$ and\/ $\span D=\Y$, then\/
${\span\!\top_{\!E,D}=\X\otimes\Y}$.
\vskip4pt
\item{$\kern-9pt$\rm(b)}
If\/ $E$ and\/ $D$ are linearly independent, then\/ $\!\top_{\!E,D}$ is
linearly independent.
\end{description}
\vskip4pt\noi
Therefore, if\/ $E$ is a Hamel basis for\/ $\X$ and\/ $D$ is a Hamel basis
for\/ $\Y$, then\/ $\!\top_{\!E,D}\!$ is a Hamel basis for\/
${\T=\X\otimes \Y}$.
\end{theorem}

\begin{proof}
(a)
Take an arbitrary element $\digamma={\sum}_{i=1}^nx_i\otimes y_i$ from
${\X\otimes\Y}.$ If ${\span E=\X}$ and ${\span D=\Y}$ then consider any
expansion of each ${x_i\in\X}$ in terms of vectors $e_{i,j}$ from $E$ and any
expansion of each ${y_i\in\Y}$ in terms of vectors $d_{i,k}$ from $D.$ Since
${x\otimes y=\theta(x,y)}$ for each pair ${(x,y)\in\X\x\Y}$, and since
${\theta\!:\X\x\Y\to\X\otimes\Y}$ is bilinear,
\goodbreak\vskip4pt\noi
$$
\digamma={\sum}_{i=1}^n
\Big({\sum}_{j=1}^m\beta_j e_{i,j}\otimes{\sum}_{k=1}^\ell\gamma_k d_{i,k}\Big)
={\sum}_{i,j,k}^{n,m,\ell}\beta_j\gamma_k(e_{i,j}\otimes d_{i,k}).
$$
\vskip2pt\noi
Thus $\digamma$ lies in $\span\!\top_{\!E,D}.$ Then
${\X\otimes\Y\sse\span\!\top_{\!E,D}}.$ So ${\span\!\top_{\!E,D}=\X\otimes\Y}$.

\vskip6pt\noi
(b)
Let $E'\!=\{e_j\}_{j=1}^m$ and $D'\!=\{d_k\}_{k=1}^\ell$ be arbitrary nonempty
finite linearly inde\-pendent subsets of $E$ and $D$ respectively, and
consider the linear manifolds
$$
\M=\span E'\!=\span\{e_j\}_{j=1}^m\sse\X\kern-1pt
\quad\hbox{and}\quad
\N=\span D'\!=\span\{d_k\}_{k=1}^\ell\sse\Y.
$$
Set $\Z=\Le[\FF^m\!,\FF^\ell]$, identified with the linear space of all
${\ell\x m}$ matrices of entries in $\FF.$ Take
${\phi\!:\M\x\N\to\Z}$ given for ${u=\sum_{j=1}^m\beta_je_j\in\M}$
and ${v=\sum_{k=1}^\ell\gamma_kd_k\in\N}$ by
$$
\phi(u,v)=\big(\beta_j\gamma_k\big)\in\Z
\quad\;\hbox{for}\;\quad
j=1,\dots,m
\;\;\hbox{and}\;\;
k=1,\dots,\ell,
$$
where $\big(\beta_j\gamma_k\big)$ is the ${\ell\x m}$ matrix whose entries
are the products $\beta_j\gamma_k$ of the coefficients of the unique
expansion of arbitrary vectors ${u\in\M}$ and ${v\in\N}$ in terms of the
linearly independent sets $E'\!$ and $D'\!.$ It is readily verified that
${\phi\!:\M\x\N\!\to\Z}$ is a bilinear map$.$ Thus consider the linear
transformation ${\Phi\!:\M\otimes\N\!\to\Z}$ such that
$\phi={\Phi\circ\theta}$ according to axiom (b) in Definition 3.1$.$ So
$\phi({u,v})=\Phi\,(\theta({u,v}))=\Phi({u\otimes v})$ for every
$u=\sum_{j=1}^m\beta_je_j$ in $\M$ and every $v=\sum_{k=1}^\ell\gamma_kd_k$
in $\N$. In particular,
$$
\Phi(e_j\otimes d_k)=\phi(e_j,d_k)=\Pi_{j,k},
$$
where ${\Pi_{j,k}\in\Z}$ is the ${\ell\x m}$ matrix whose entry at position
${j,k}$ is 1 and all other entries are 0$.$ These matrices form a linearly
independent set in $\Z.$ (In fact, $\{\Pi_{j,k}\}_{j,k=1}^{m,\ell}\!$ is the
canonical Hamel basis for $\Z$.) Take any pair of integers ${k',j'}$ and
suppose ${e_{j'}\otimes d_{k'}}$ is a linear combination of the remaining
single tensors $\{{e_j\otimes d_k}\}_{j,k\in I'}$ with
$I'=\{{j,k=1\;\hbox{to}\;\,m,\ell}\!:\,{j\ne j',k\ne k'}\}$,
say
$$
e_{j'}\otimes d_{k'}
={\sum}_{j,k\in I'}\delta_{j,k}\kern1pt(e_j\otimes d_k).
$$
Then, as ${\Phi\!:\M\otimes\N\to\Z}$ is linear,
$$
\Pi_{j',k'}=\Phi(e_{j'}\otimes d_{k'})
={\sum}_{j,k\in I'}\delta_{j,k}\kern1pt\Phi(e_j\otimes d_k)
={\sum}_{j,k\in I'}\delta_{j,k}\kern1pt\Pi_{j,k},
$$
and so ${\delta_{j,k}=0}$ for every ${j,k\in I'}$ since
$\{\Pi_{j,k}\}_{j,k=1}^{m,\ell}\!$ is linearly independent in $\Z.$ Hence
$\{e_j\otimes d_k\}_{j,k=1}^{m,\ell}\!$ is linearly independent in
$\top_{\!E,D}.$ In other words,
$$
\top_{\!E'\!,D'}\!
=\{x\otimes y\in\X\otimes\Y\!:\,x\in E'\;\hbox{and}\;\,y\in D'\}
$$
is a finite linearly independent subset of $\top_{\!E,D}$ whenever $E'$ and
$D'$ are finite linearly independent subsets of $E$ and $D.$ Thus if
$E$ and $D$ are linearly independent subsets of $\X$ and $\Y$, then every
finite subset of each of them is trivially linearly independent, and so is
every finite subset of $\top_{\!E,D}$ as we saw above$.$ But if every finite
subset of $\top_{\!E,D}$ is linearly independent, then so is
$\top_{\!E,D}$ (see, e.g., \cite[Proposition 2.3]{EOT}).
\end{proof}

\begin{corollary}
\qquad
$\dim(\X\otimes\Y)=\dim\X\cdot\dim\Y$.
\end{corollary}

\begin{proof}
$\!$If $E$ and $D$ are Hamel basis for $\X$ and $\Y$, then $\top_{\!E,D}$ is
a Hamel basis for $\T$ by Theorem 3.2$.$ Also
$\!\top_{\!E,D}
={\big\{x\otimes y\in\X\otimes\Y\!:\,x\in E\;\hbox{and}\;\,y\in D\big\}}$ is
in a one-to-one correspondence with ${E\x D}$ as $E$ and $D$ are linearly
independent$.$ Since $\#({E\x D})=$ ${\#E\cdot\#D}$ by definition of product
of cardinal numbers (see, e.g.,
 \cite[Pro\-blem 1.30]{EOT}), then
$\#\top_{\!E,D}\kern-1pt={\#E\cdot\#D}.$ Thus follows the claimed dimension
identity.
\end{proof}

A straightforward consequence of the dimension identity of Corollary 3.2 is
this$.$ Tensor product is commutative up to an isomorphism:
$$
\X\otimes\Y\cong\Y\otimes\X.
$$
\vskip-2pt

\vskip4pt
Next we identify a special type of linear manifold of a tensor product
space.

\begin{proposition}
Let\/ $\X\kern-1pt$ and\/ $\kern-1pt\Y$ be linear spaces$.$ Suppose\/ $\M$
and\/ $\N$ are \hbox{linear} manifolds of\/ $\X\kern-1pt$ and\/ $\kern-1pt\Y$
respectively$.$ Let\/ ${(\X\otimes\kern-.5pt\Y,\theta)}$ be a tensor
product$.$ Set\/ ${\M\otimes\N}\kern-1pt=$ $\span R(\theta|_{\M\x\N}).$ Then\/
${\M\otimes\N}$ is a linear manifold of\/ the tensor product space
${\X\otimes\Y}$ and\/ ${(\M\otimes\N,\vtheta)}$ is a tensor product with\/
${\vtheta=\theta|_{\M\x\N}}$.
\end{proposition}

\proof
Consider Definition 3.1$.$ Let ${(\X\otimes\Y,\theta)}$ be a tensor
product$.$ Take any biline\-ar map ${\phi\!:\X\x\Y\to\Z}.$ Let
${\Phi\!:\X\otimes\Y\to\Z}$ be the linear transformation such that
$$
\phi=\Phi\circ\theta.
$$
Let $\M$ and $\N$ be linear manifolds of $\X$ and $\Y.$ Take the restriction
$\theta|_{\M\x\N}$ of the natural bilinear map $\theta$ to the Cartesian
product ${\M\x\N\sse\X\x\Y}$ so that
$$
\theta|_{\M\x\N}\!:\M\x\N\to R(\theta|_{\M\x\N})
\sse\span R(\theta|_{\M\x\N})\sse\span R(\theta)=\X\otimes\Y,
$$
which is a bilinear map$.$ Now consider the restriction $\phi|_{\M\x\N}$ of
the arbitrary bilinear map ${\phi\!:\X\x\Y\to\Z}$ to ${\M\x\N}$, which is
again a bilinear map for which
$$
\phi|_{\M\x\N}
=(\Phi\circ\theta)|_{\M\x\N}
=\Phi\circ\theta|_{\M\x\N}
=\Phi|_{\span\!R(\theta|_{\M\x\N})}\circ\theta|_{\M\x\N},
$$
where ${\Phi|_{\span\!R(\theta|_{\M\x\N})}}$ is the restriction of the linear
transformation ${\Phi\!:\X\otimes\Y\to\Z}$ to the linear manifold
$\span R(\theta|_{\M\x\N})$, again a linear transformation$.$ Proposition 2.2
says that every bilinear map ${\psi\!:\M\x\N\to\Z}$ is of the form
${\phi|_{\M\x\N}\!:\M\x\N\to\Z}$ for some bilinear map ${\phi\!:\X\x\Y\to\Z}.$
Thus for every bilinear map ${\psi\!:\M\x\N\to\Z}$ there is a linear
transformation
${\Phi|_{\span\!R(\theta|_{\M\x\N})}\!:\span R(\theta|_{\M\x\N})\to\Z}$ such
that
$$
\phi=\Phi|_{\span\!R(\theta|_{\M\x\N})}\circ\theta|_{\M\x\N}.
$$
Set $\M\otimes\N=\span R(\theta|_{\M\x\N})\sse\X\otimes\Y$, a linear manifold
of the linear space ${\X\otimes\Y}$, and $\vtheta=\theta|_{\M\x\N}$, the
bilinear restriction of the bilinear $\theta.$ Thus by Definition 3.1
$$
\hbox{${(\M\otimes\N,\vtheta)}$ is a tensor product}. \eqno{\qed}
$$
\vskip-2pt

\vskip6pt
Therefore ${\M\otimes\N}$ stands for {\it the}\/ tensor product space of
linear manifolds $\M$ and $\N$ of the linear spaces $\X$ and $\Y$ according to
Corollary 3.1 and Proposition 3.4.

\begin{definition}
A linear manifold $\Upsilon\!$ of a tensor product space ${\T=\X\otimes\Y}$ is
{\it regular}\/ if ${\Upsilon\!=\M\otimes\N}$ for some linear manifolds $\M$
and $\N$ of the linear spaces $\X$ and $\Y.$ Otherwise $\Upsilon$ is called
{\it irregular}\/.
\end{definition}

The next characterization of regular linear manifolds is straightforward from
Theorem 3.2$.$ For a collection of properties of regular linear manifolds see
\cite{Kub2, Kub3}.

\begin{proposition}
{\it A nonzero linear manifold\/ $\Upsilon\kern-1pt$ of\/ ${\X\otimes\Y}$ is
regular if and only if
$$
\Upsilon=\span\!\top_{\!E'\!,D'}
$$
for some nonempty subsets\/ ${E'\sse E}$ and\/ ${D'\sse D}$ for Hamel bases\/
$E$ and\/ $D$ for\/ $\X$ and\/ $\Y$ respectively}\/.
\end{proposition}

The concept of tensor product of linear transformations is given as follows.

\begin{definition}
Let ${\X,\Y,\V,\W}$ be linear spaces and consider the tensor product spaces
${\X\otimes\Y}$ and ${\V\otimes\W}.$ Let ${A\in\Le[\X,\V]}$ and
${B\in\Le[\Y,\W]}$ be linear trans\-for\-ma\-tions$.$ For each
${\sum_{i=1}^nx_i\otimes y_i}$ in ${\X\otimes\Y}$ set
$$
(A\otimes B)\hbox{$\sum$}_{i=1}^nx_i\otimes y_i
=\hbox{$\sum$}_{i=1}^nAx_i\otimes By_i
\:\;\hbox{in}\;\;
\V\otimes\W.
$$
This defines a map ${A\otimes B}$ of the linear space ${\X\otimes\Y}$ into the
linear space ${\V\otimes\W}$, which is referred to as the {\it tensor product
of the transformations\/ $A$ and}\/ $B$, or the {\it tensor product
transformation}\/ ${A\otimes B}$.
\end{definition}

\begin{proposition}
Take\/ ${A\in\Le[\X,\V]}$ and\/ ${B\in\Le[\Y,\W]}$.
\vskip6pt\noi
{\rm(a)}
In fact,\/ ${A\otimes B}$ in Definition 3.3 defines a linear transformation\/,
$$
A\otimes B\in\Le[\X\otimes\Y,\V\otimes\W],
$$
and\/ $(A\otimes B)\digamma$ does not depend on the representation\/
${{\sum}_{i=1}^nx_i\!\otimes\!y_i}$ of\/ ${\digamma\!\in\X\otimes\Y}$.
\vskip6pt\noi
{\rm(b)}
The map\/ $\theta\!:\Le[\X,\V]\x\Le[\Y,\W]\to\Le[\X\otimes\Y,\V\otimes\W]$
defined by
$$
\theta(A,B)=A\otimes B
\quad\;\hbox{for every}\;\quad
(A,B)\in\Le[\X,\V]\x\Le[\Y,\W],
$$
with\/ $A\otimes B\in\Le[\X\otimes\Y,\V\otimes\W]$ as in\/ $(\hbox{\rm a})$,
is bilinear\/.
\vskip6pt\noi
{\rm(c)}
Set
\vskip-4pt\noi
$$
\Le[\X,\V]\otimes\Le[\Y,\W]
\,=\,\span R(\theta)\,\sse\,\Le[\X\otimes\Y,\V\otimes\W].
$$
\vskip4pt\noi
Then\/ ${(\Le[\X,\V]\otimes\Le[\Y,\W],\theta)}$ is a tensor product of\/
${\Le[\X,\V]}$ and\/ ${\Le[\Y,\W]}$.
\vskip6pt\noi
{\rm(d)}
The transformation ${A\otimes B\in\Le[\X\otimes\Y,\V\otimes\W]}$ in\/
$(\hbox{\rm a})$ coincides with a single tensor in the tensor product space\/
${\Le[\X,\V]\otimes\Le[\Y,\W]}\!:$
$$
A\otimes B\,\in\,\Le[\X,\V]\otimes\Le[\Y,\W]
\,\sse\,\Le[\X\otimes\Y,\V\otimes\W].
$$
\end{proposition}

\begin{proof}
$\!$Items (a), (b), (d) are readily verified$.$ Item (c) goes as follows$.$
Note that
\begin{eqnarray*}
\span R(\theta)
&\kern-7pt=\kern-7pt&
\span\big\{A\otimes B\in\Le[\X\otimes\Y,\V\otimes\W]
\!:A\in\Le[\X,\V]\;\,B\in\Le[\Y,\W]\big\}                                 \\
&\kern-7pt=\kern-7pt&
\!\big\{\hbox{$\sum$}_{i=1}^n
A_i\!\otimes\kern-1ptB_i\!\in\!\Le[\X\otimes\Y,\V\otimes\W]
\!:A\!\in\!\Le[\X,\V],\,B\!\in\!\Le[\Y,\W],\,n\!\in\!\NN]\big\}.
\end{eqnarray*}
For each bilinear map ${\phi\!:\Le[\X,\V]\x\Le[\Y,\W]\to\Z}$ take
${\Phi\!:\span R(\theta)\to\Z}$ defined by
$$
\Phi\Big({\sum}_{i=1}^n\!A_i\otimes B_i\Big)
={\!\sum}_{i=1}^n\!\phi(A_i,B_i)\,\in\,\Z
\quad\hbox{for every}\quad
{\sum}_{i=1}^n\!A_i\otimes B_i\in\span R(\theta).
$$
Since $\phi$ is bilinear, it is easy to show that $\Phi$ is linear:
${\Phi\in\Le[\span R(\theta),\Z]}.$ Moreover, for every
${(A,B)\in\Le[\X,\V]\x\Le[\Y,\W]}$
$$
(\Phi\circ\theta)(A,B)=\Phi\big(\theta(A,B)\big)=\Phi(A\otimes B)=\phi(A,B).
$$
Thus $\phi=\Phi\circ\theta$, equivalently, the diagram
\vskip0pt\noi
$$
\newmatrix{
\Le[\X,\V]\x\Le[\Y,\W] & \kern2pt\buildrel\phi\over\emap & \kern-1pt\Z   \cr
       &                                 &                               \cr
       & \kern-3pt_\theta\kern-3pt\semap & \kern4pt\nmap\scriptstyle\Phi \cr
       &                                 & \phantom{;}                   \cr
       &                                 & \kern-6pt\span R(\theta)      \cr}
$$
commutes. Therefore $(\span R(\theta),\theta)$ satisfies the axioms of
Definition 3.1.
\end{proof}

The particular case of ${\V=\W=\FF}$ is worth noticing$.$ In this case
$$
\Le[\X,\V]=\Le[\X,\FF]=\X^\sharp,
\quad
\Le[\Y,\W]=\Le[\Y,\FF]=\Y^\sharp
\quad\hbox{and}
$$
$$
\Le[\X\otimes\Y,\V\otimes\W]=\Le[\X\otimes\Y,\FF\otimes\FF]
=\Le[\X\otimes\Y,\FF]=(\X\otimes\Y)^\sharp
$$
Indeed, $\dim({\FF\otimes\FF})\kern-1pt=\kern-1pt\dim\FF\kern-1pt=\kern-1pt1$
by Corollary 3.2. So write ${\FF\otimes\FF}\kern-1pt=\kern-1pt\FF$
for ${\FF\otimes\FF}\kern-1pt\cong\kern-1pt\FF\kern-.5pt$ as usual$.$ Since
${\Le[\X,\V]\otimes\Le[\Y,\W]\sse\Le[\X\otimes\Y,\V\otimes\W]}$ by
Proposition 3.6(c), then
$$
\X^\sharp\otimes\Y^\sharp\sse(\X\otimes\Y)^\sharp.
$$
\vskip-2pt

\vskip4pt
Basic results on tensor product transformations are given next$.$ Most are
straightforward or readily verified: properties (a,b) are trivial since
${A\otimes B}$ is a single tensor, (c,d) are straightforward by definition of
${A\otimes B}$, and (e,f) are readily verified for the regular linear
manifolds ${\N(A)\otimes\N(B)}$ and ${\R(A)\otimes\R(B)}.$ For the
nonreversible inclusion in (f) see, e.g., \cite{KD}$.$ Item (g) says$:$
tensor product of linear transformations is commutative up to isomorphisms,
which means ${A\otimes B}$ and ${B\otimes A}$ are {\it isomorphically
equivalent}\/ in the sense that ${\Pi_2(A\otimes B)=(B\otimes A)\Pi_1}$ for
isomorphisms ${\Pi_1\!:\X\otimes\Y\to\Y\otimes\X}$ and
${\Pi_2\!:\V\otimes\W\to\W\otimes\V}$ (whose existence follows from the fact
that ${\X\otimes\Y}\cong{\Y\otimes\X}$ as a consequence of Corollary 3.2)$.$
We prove (h) below.

\begin{proposition}
Let\/ ${\V,\W,\X,\Y,\X',\Y'}$ be linear spaces$.$ Take\/
${A,A_1,A_2\in\Le[\X,\V]}$\/ $\,{B,B_1,B_2\!\in\Le[\Y,\W]}$,
$\,{C\in\Le[\X',\X]}$, $\,{D\in\Le[\Y',\Y]}$ and also ${\alpha,\beta\in\FF}.$
Then
\begin{description}
\item{\rm(a)}
$\;{\alpha\,\beta\,(A\otimes B)}
={\alpha A\otimes \beta B}
={\alpha\,\beta A\otimes B}
={A\otimes \alpha\,\beta B}$,
\vskip4pt
\item{\rm(b)}
$\;{(A_1+A_2)\otimes (B_1+B_2)}
={A_1\otimes B_1}+{A_2\otimes B_1}
+{A_1\otimes B_2}+{A_2\otimes B_2}$,
\vskip6pt
\item{\rm(c)}
$\;{A\kern1ptC\otimes BD}={(A\otimes B)\,(C\otimes D)}$,
\vskip4pt
\item{\rm(d)}
$\;$If\/ $A$ and\/ $B$ are invertible, then so is\/
${A\otimes\kern-1pt B}$ and\/
${(A\otimes\kern-1pt B)^{-1}}
\kern-1pt=\kern-1pt{A^{-1}\kern-1pt\otimes\kern-1pt B^{-1}}$,
\vskip4pt
\item{\rm(e)}
$\;{\R(A)\otimes\R(B)=\R(A\otimes B)}$,
\vskip4pt
\item{\rm(f)}
$\;{\N(A)\otimes\N(B)\ssen\N(A\otimes B)}$,
\vskip4pt
\item{\rm(g)}
$\;{A\otimes B\cong B\otimes A}$.
\vskip4pt
\item{\rm(h)}
$\;{(A\otimes B)^\sharp}={A^\sharp\otimes B^\sharp\!}$.
\end{description}
\end{proposition}

\begin{proof}
(h) Take ${A\kern-1pt\in\kern-1pt\Le[\X,\V]}$,
${B\kern-1pt\in\kern-1pt\Le[\Y,\W]}$,
${A^\sharp\kern-1pt\in\kern-1pt\Le[\V^\sharp,\X^\sharp]}$,
${B^\sharp\kern-1pt\in\kern-1pt\Le[\W^\sharp,\Y^\sharp]}.$ Consider the single
tensors ${A\otimes B}$ in
${\Le[\X,\V]\otimes\Le[\Y,\W]}\sse{\Le[\X\otimes\Y,\V\otimes\W]}$ and
${A^\sharp\otimes B^\sharp}$ in
${\Le[\V^\sharp,\X^\sharp]\otimes\Le[\W^\sharp,\Y^\sharp]}
\sse{\Le[\V^\sharp\otimes\W^\sharp,\X^\sharp \otimes\Y^\sharp]}$,
and the algebraic adjoint
${(A\otimes B)^\sharp}$ in ${\Le[(\Y\otimes\W)^\sharp,(\X\otimes\V)^\sharp]}.$
Take an arbitrary 
${f\otimes g\in\V^\sharp\otimes\W^\sharp}\sse(\V\otimes\Y)^\sharp.$ By
definition of tensor product transformation and of algebraic adjoint
$$
(A^\sharp\otimes B^\sharp)(f\otimes g)
=A^\sharp f\otimes B^\sharp g
=f\kern-1ptA\otimes gB
\;\in\;\X^\sharp\otimes\Y^\sharp\sse(\X\otimes\Y)^\sharp,    \eqno(*)
$$
$$
(A\otimes B)^\sharp(f\otimes g)=(f\otimes g)(A\otimes B)
\;\in\;
(\X\otimes\Y)^\sharp.                                        \eqno(**)
$$

\vskip6pt\noi
{\it Claim}\/$.$
\qquad
$(f\otimes g)(A\otimes B)
=f\kern-1ptA\otimes gB
\;\in\;\X^\sharp\otimes\Y^\sharp\sse(\X\otimes\Y)^\sharp.$

\vskip6pt\noi
{\it Proof}\/$.$
Take an arbitrary single tensor ${x\otimes y}$ in ${\X\otimes\Y}$, and an
arbitrary single ten\-sor ${A\otimes B}$ in ${\Le[\X,\V]\otimes\Le[\Y,\W]}.$
By definition of tensor product transformation,
${(A\otimes B)(x\otimes y)}={Ax\otimes By}$, a single tensor in
${\V\otimes\W}.$ Next take an arbitrary single tensor ${f\otimes g}$ in
${\V^\sharp\otimes\W^\sharp}={\Le[\V,\FF]\otimes\Le[\W,\FF]}$ so that, by
definition of tensor product transformation,
${(f\otimes g)(Ax\otimes By)}={fAx\otimes gBy\in\FF\otimes\FF=\FF}.$
On the other hand, since ${fA\in\X^\sharp}=\Le[{\X,\FF}]$ and
${gB\in\Y^\sharp}=\Le[{\Y,\FF}]$, then (definition of tensor product
transformation),
${(fA\otimes gB)(x\otimes y)}={fAx\otimes gBy\in\FF\otimes\FF=\FF}$.
Summing up:
$$
(f\otimes g)(A\otimes B)(x\otimes y)
=(f\otimes g)(Ax\otimes By)
=f\kern-1ptAx\otimes gBy
=(f\kern-1ptA\otimes gB)(x\otimes y)
$$
for every ${x\otimes y\in\X\otimes\Y}.$ Thus
$({f\otimes g})({A\otimes B})\digamma
=({f\otimes g})({A\otimes B})\sum_i{x_i\otimes y_i}
=\sum_i({f\otimes g})({A\otimes B})({x_i\otimes y_i})
=\sum_i({f\kern-1ptA\otimes gB})({x_i\otimes y_i})
=({f\kern-1ptA\otimes gB})\sum_i{x_i\otimes y_i}
=({f\kern-1ptA\otimes gB})\digamma$
for every ${\digamma\in\X\otimes\Y}$, and hence
$$
(f\otimes g)(A\otimes B)={f\kern-1ptA\otimes gB}
$$
in ${\X^\sharp\otimes\Y^\sharp}={\Le[{\X,\FF]\otimes\Le[\Y,\FF}]}
\sse\Le[{\X\otimes\Y,\FF}]=({\X\otimes\Y)^\sharp}.\!\!\!\qed$

\vskip6pt\noi
Then by $(*)$, $(**)$ and the above claim
\goodbreak\vskip4pt\noi
$$
(A^\sharp\otimes B^\sharp)(f\otimes g)=(A\otimes B)^\sharp(f\otimes g)
$$
\vskip2pt\noi
for every single tensor ${f\otimes g\in\X^\sharp\otimes\Y^\sharp}.$ Thus by
a similar argument
$$
A^\sharp\otimes B^\sharp=(A\otimes B)^\sharp
$$
in $\Le[{\X^\sharp\otimes\Y^\sharp,\FF}]=({\X^\sharp\otimes\Y^\sharp})^\sharp$,
concluding the proof of (g).
\end{proof}

\section{An Interpretation via Quotient Space}

Let $\X$ and $\Y$ be nonzero linear spaces over a field $\FF.$ Take the
Cartesian product $S={\X\x\Y}$ of $\X$ and $\Y.$ Consider the notation and
terminology of Subsection 2.1$.$ Thus $\SSe$ is the free linear space
generated by $S$ (i.e., the linear space of all functions ${f\!:\X\x\Y\to\FF}$
that vanish everywhere on the complement of some finite subset of ${\X\x\Y}$),
and \hbox{\small$\SSs=$} ${\{e_{(x,y)}\}_{(x,y)\in S}}$ is the Hamel basis for
$\SSe$ consisting of characteristic functions
$e_{(x,y)}=\vchi_{\{(x,y)\}}={\vchi_{\{x\}}\vchi_{\{y\}}\!:{\X\x\Y}\to\FF}$
of all singletons at each pair of vectors ${(x,y)}$ in the Cartesian product
$S={\X\x\Y}.$ With the identifi\-cation $\approx$ of Subsection 2.1 still in
force, take the sums of elements from \hbox{\small$\SSs$} whose double indices
have one common entry and consider the following differences$.$
\begin{description}
\item{\rm(i)}\quad
$e_{(x_1+x_2,y)}-e_{(x_1,y)}-e_{(x_2,y)}
\,\approx\,
(x_1\!+x_2\,,y)-(x_1\,,y)-(x_2\,,y)$,
\vskip4pt
\item{$\kern-1pt$\rm(ii)}\quad
$e_{(x,y_1+x_2)}-e_{(x,y_1)}-e_{(x,y_2)}
\,\approx\,
(x\,,y_1\!+y_2)-(x\,,y_1)-(x\,,y_2)$,
\vskip4pt
\item{$\kern-3pt$\rm(iii)}\quad
$e_{(\alpha x,y)}-\alpha\,e_{(x,y)}
\,\approx\,
(\alpha x,y)-\alpha(x,y)$,
\vskip4pt
\item{$\kern-3pt$\rm(iv)}\quad
$e_{(x,\alpha y)}-\alpha\,e_{(x,y)}
\,\approx\,
(x,\alpha y)-\alpha(x,y)$,
\end{description}
for every ${x,x_1,x_2\in\X}$, every ${y,y_1,y_2\in\Y}$, and every
${\alpha\in\FF}.$ The above differences are not null$.$ If they were, then we
could identify a bilinear rule on the ordered pairs ${(x,y)\in\X\x\Y}.$ Thus
look at equivalence classes $[e_{(x,y)}]$ of characteristic functions
$e_{(x,y)}\!$ in \hbox{\small$\SSs$} $\sse\SSe$, gathering those differences at
the origin of a quotient space as fol\-lows$.$ Take the linear manifold $\M\!$
of $\SSe$ generated by the differences in \hbox{(i)--(iv), viz.,}
\begin{eqnarray*}
\M=\span
\kern-13pt&\big\{\kern-13pt&
e_{(x_1+x_2,y)}-e_{(x_1,y)}-e_{(x_2,y)}\,, \quad
e_{(x,y_1+x_2)}-e_{(x,y_1)}-e_{(x,y_2)}\,,                               \\
&&
e_{(\alpha x,y)}-\alpha\,e_{(x,y)}\,, \quad
e_{(x,\alpha y)}-\alpha\,e_{(x,y)}\,\big\}.
\end{eqnarray*}
\vskip2pt\noi
Take the quotient space $\SSe/\M$ of $\SSe$ modulo $\M$, consider the natural
quotient map
$$
\pi\!:\SSe\to\SSe/\M,
$$
\vskip-4pt\noi
and define a map
$$
\theta\!:\X\x\Y\to\SSe/\M
$$
on ${S=\X\x\Y}$ as follows$.$ For every ${(x,y)\in\X\x\Y}$ set
$\theta(x,y)=\pi(e_{(x,y)}).$ Therefore
$$
\theta(S)=\pi(\hbox{\small$\SSs$}).
$$
The Hamel basis $\hbox{\small$\SSs=\,$}\{e_{(x,y)}\}_{(x,y)\in S}\kern-.5pt$
for the linear space $\SSe$ generated by $S$ is naturally identified with $S$
itself, and so the domain of $\theta$ is identified with the domain of the
restriction $\pi|_{\kern1pt\SSs}$ of the natural quotient map $\pi$ to
\hbox{\small$\SSs$}, that is, ${S\approx\hbox{\small$\SSs$}}.$ Since they also
coincide pointwise, then $\theta$ is naturally identified with
$\pi|_{\kern1pt\SSs}.$ So write
$$
\theta=\pi|_{\kern1pt\SSs}
\quad\;\hbox{for}\;\quad
\theta\approx\pi|_{\kern1pt\SSs}.
$$
Elements of $\SSe/\M$ which are images of the map
${\theta\!:\X\x\Y\to\SSe/\M}$ are denoted by ${x\otimes y}$, and again
referred to as {\it single tensors}\/ or {\it decomposable elements}\/:
$$
x\otimes y=\theta(x,y)=\pi(e_{(x,y)})=[e_{(x,y)}]
\quad\;\hbox{for every}\;\quad
(x,y)\in\X\x\Y.
$$

\begin{theorem}
\quad
${(\SSe/\M,\theta)}$ is a tensor product of\/ $\X$ and\/ $\Y$\/.
\end{theorem}

\begin{proof}
Consider the axioms (a) and (b) in Definition 3.1.
\goodbreak\noi

\vskip6pt\noi
(a$_1$)
By definition of $\M$, the differences in (i) to (iv) lie in
${\M=[0]\in\SSe/\M}.$ Thus with $\theta(x,y)=\pi(e_{(x,y)})=[e_{(x,y)}]$
it follows that
$\theta\!:\X\x\Y\to\SSe/\M$ is a bilinear map.

\vskip6pt\noi
(a$_2$)
Since \hbox{\small$\SSs=$} $\{e_{(x,y)}\}_{(x,y)\in\X\x\Y}$ is a Hamel basis
for the linear space $\SSe$, then ${\span\pi(\hbox{\small$\SSs$})=\SSe/\M}$
(cf$.$ Remark 2.1). Thus as $R(\theta)=R(\pi|_{\kern1pt\SSs})$,
$$
\span R(\theta)=\span R(\pi|_{\kern1pt\SSs})
=\span\pi(\hbox{\small$\SSs$})=\SSe/\M.
$$
\vskip-2pt

\vskip6pt\noi
(b) Take a bilinear map ${\phi\!:\X\x\Y\to\Z}$ into a linear space $\Z$ and
consider a transformation $\wtil\Phi$ on the free linear space $\SSe$
generated by the Cartesian product ${\X\x\Y}$,
$$
\wtil\Phi\!:\SSe\to\Z,
$$
\vskip-6pt\noi
defined by
\vskip4pt\noi
$$
\wtil\Phi(f)={\sum}_{i=1}^n\alpha_i\,\phi(x_i,y_i)\,\in\,\Z
\quad\;\hbox{for every}\;\quad
f={\sum}_{i=1}^n\alpha_ie_{(x_i,y_i)}\in\SSe,
$$
which is clearly linear$.$ Moreover, since
${\wtil\Phi(e_{(x_i,y_i)})=\phi(x,y)}$, then
$$
\wtil\Phi(\hbox{\small$\SSs$})=\phi(S).
$$
Again, since ${\hbox{\small$\SSs$}\approx S=\X\x\Y}$, then
$\wtil\Phi|_{\kern1pt\SSs}$ is naturally identified with $\phi.$ So write
$$
\wtil\Phi|_{\kern1pt\SSs}=\phi.
$$
Furthermore, since ${\phi\!:\X\x\Y\to\Z}$ is bilinear, then $\wtil\Phi$
evaluated at the differences in (i) to (iv) is null$.$ Hence the linear
manifold $\M$ of $\SSe$ is such that ${\wtil\Phi(\M)=0}.$ That is,
${\M\sse\N(\wtil\Phi)}.$ Thus by Proposition 2.1 there exists a unique linear
transformation
$$
\Phi\!:\SSe/\M\to\Z
$$
such that ${\wtil\Phi=\Phi\circ\pi}.$ Therefore restricting to
\hbox{\small$\SSs$} ${\subset\SSe}$ and since
$\span\pi(\hbox{\small$\SSs$})=\SSe/\M$, we get
$\phi=\wtil\Phi|_{\kern1pt\SSs}=(\Phi\circ\pi)|_{\kern1pt\SSs}
=\Phi|_{\span\!\R(\pi)}\circ\pi|_{\kern1pt\SSs}=\Phi\circ\theta.$
Equivalently, the diagrams
\vskip0pt\noi
$$
\newmatrix{
\SSe & \kern2pt\buildrel\wtil\Phi\over\emap & \kern-1pt\Z                \cr
     &                                   &                               \cr
     & \kern-3pt_\pi\kern-3pt\semap      & \kern4pt\nmap\scriptstyle\Phi \cr
     &                                   & \phantom{;}                   \cr
     &                                   & \kern-2pt\SSe/\M              \cr}
\qquad\hbox{and}\qquad
\newmatrix{
\hbox{\small$\SSs$} & \kern2pt\buildrel\phi\over\emap & \kern-1pt\Z       \cr
     &                                   &                               \cr
     & \kern-3pt_\theta\kern-3pt\semap   & \kern4pt\nmap\scriptstyle\Phi \cr
     &                                   & \phantom{;}                   \cr
     &                                   & \kern-2pt\SSe/\M              \cr}
$$
commute$.$ Identifying again ${S\approx\hbox{\small$\SSs$}}$ (thus regarding
$S=\X\x\Y$ as a subset of $\SSe$ and writing
$\theta\kern-1pt=\kern-1pt\pi|_S\kern-1pt$ for
$\theta\kern-1pt=\kern-1pt\pi|_{\kern1pt\SSs}\kern-1pt$ and
$\phi=\wtil\Phi|_S\kern-1pt$ for
$\phi\kern-1pt=\kern-1pt\wtil\Phi|_{\kern1pt\SSs}\kern-1pt$), then the
diagram
\vskip0pt\noi
$$
\newmatrix{
\X\x\Y & \kern2pt\buildrel\phi\over\emap & \kern-1pt\Z                   \cr
       &                                 &                               \cr
       & \kern-3pt_\theta\kern-3pt\semap & \kern4pt\nmap\scriptstyle\Phi \cr
       &                                 & \phantom{;}                   \cr
       &                                 & \kern-2pt\SSe/\M.             \cr}
$$
commutes$.$ Therefore the pair ${(\SSe/\M,\theta)}$ satisfies the axioms of
Definition 3.1.
\end{proof}

\section{An Interpretation via Linear Map of Bilinear Maps}

Let $\X$ and $\Y$ be nonzero linear spaces over the same field $\FF$, take
the Cartesian product ${\X\x\Y}$ of $\X$ and $\Y$, and consider the linear
space $b[\X\x\Y,\F]$ of all bilinear maps into an arbitrary but fixed linear
space $\F$ over $\FF.$ Associated with each pair ${(x,y)\in\X\x\Y}$, consider
a transformation ${x\otimes y\!:b[\X\x\Y,\F]\to\F}$ defined by
$$
(x\otimes y)(\psi)=\psi(x,y)\,\in\,\F
\quad\;\hbox{for every}\;\quad
\psi\in b[\X\x\Y,\F].
$$
This again is referred to as a {\it single tensor}\/ and as is readily
verified ${x\otimes y}$ is a linear transformation on the linear space of
bilinear maps,
$$
x\otimes y\in\Le[\,b[\X\x\Y,\F],\F].
$$
Thus the term {\it linear maps of bilinear maps}\/ means that this approach
to tensor product focuses explicitly on the linearization of bilinear maps$.$
Take the collection
$$
\top_{\!\X,\Y,\,\F}=\big\{x\otimes y\in\Le[\,b[\X\x\Y,\F],\F]
\!:\,x\in\X\;\hbox{and}\;\,y\in\Y\big\}
$$
of all single tensors$.$ Consider its span,
${\span\!\top_{\!\X,\Y,\,\F}\sse\Le[\,b[\X\x\Y,\F],\F]}$, which is a linear
manifold of the linear space $\Le[\,b[\X\x\Y,\F],\F]$, and define a map
$$
\theta\!:\X\x\Y\to\!\top_{\!\X,\Y,\,\F}\sse\span\!\top_{\!\X,\Y,\,\F}
$$
as follows: for each pair ${(x,y)\in\X\x\Y}$ set
$$
\theta(x,y)=x\otimes y.
$$

\begin{theorem}
\quad
${(\span\!\top_{\!\X,\Y,\,\F},\theta)}$ is a tensor product of\/ $\X$ and\/
$\Y$\/.
\end{theorem}

\begin{proof}
Take an arbitrary linear space $\F$. Consider axioms (a), (b) in
Definition 3.1.

\vskip6pt\noi
(a$_1$)
$\theta(x,y)$ is a linear transformation,
${\theta(x,y)\in\Le[b[\X\x\Y,\F],\F]}$ for each ${(x,y)}$ in ${\X\x\Y}$,
which is given by ${\theta(x,y)(\psi)=\psi(x,y)\in\F}$ for every
${\psi\in b[\X\x\Y,\F]}.$ Then bilinearity of $\psi$ is transferred to
$\theta.$ Hence ${\theta\in b[\X\x\Y,\span\!\top_{\!\X,\Y,\,\F}]}$.

\vskip6pt\noi
(a$_2$)
Since ${\theta(x,y)=x\otimes y}$, then $R(\theta)=\top_{\!\X,\Y,\,\F}$, and
so ${\span R(\theta)=\span\!\top_{\!\X,\Y,\,\F}}$.

\vskip6pt\noi
(b)
An arbitrary element $\digamma$ of the linear space
$\span\!\top_{\!\X,\Y,\,\F}$ is a linear combination of single
tensors, thus lying in $\Le[\,b[\X\x\Y,\F],\F].$ Since $\theta$ is bilinear,
then every $\digamma$ in $\span\!\top_{\!\X,\Y,\,\F}$ is a finite sum of
single tensors ${x\otimes y=\theta(x,y)}$:
$$
\digamma={\sum}_ix_i\otimes y_i\,\in\,
\span\!\top_{\!\X,\Y,\,\F}\,\sse\,\Le[\,b[\X\x\Y,\F],\F].
$$
Given an arbitrary bilinear map ${\phi\in b[\X\x\Y,\Z]}$ into any linear space
$\Z$, consider the transformation
${\Phi\!:\span\!\top_{\!\X,\Y,\,\F}\to\Z}$ defined by
$$
\Phi(\digamma)={\sum}_i\phi(x_i,y_i)\,\in\,\Z
\quad\;\hbox{for every}\;\quad
\digamma={\sum}_ix_i\otimes y_i\,\in\,\span\!\top_{\!\X,\Y,\,\F}.
$$
As is readily verified, this is a linear transformation:
${\Phi\in\Le[\kern1pt\span\!\top_{\!\X,\Y,\,\F},\Z]}.$ Also,
$$
(\Phi\circ\theta)(x,y)=\Phi\big(\theta(x,y)\big)=\Phi(x\otimes y)=\phi(x,y)
$$
for every ${(x,y)\in\X\x\Y}.$ Hence $\phi=\Phi\circ\theta$, leading to the
commutative diagram
\vskip0pt\noi
$$
\newmatrix{
\X\x\Y & \kern2pt\buildrel\phi\over\emap & \kern-1pt\Z                    \cr
  &                                 &                                     \cr
  & \kern-3pt_\theta\kern-3pt\semap & \kern4pt\nmap\scriptstyle\Phi       \cr
  &                                 & \phantom{;}                         \cr
  &                                 & \kern-9pt\span\!\top_{\!\X,\Y,\,\F} \cr}
$$
and therefore the pair $(\span\!\top_{\!\X,\Y,\,\F},\theta)$ satisfies the
axioms of Definition 3.1.
\end{proof}

This interpretation, where single tensors are defined as linear
transformations of bilinear maps, is specially tailored to highlight the
central property of tensor products as a tool to linearize bilinear maps
according to Proposition 3.3$.$

\vskip4pt
An important particular case refers to linear and bilinear forms by setting
$\F=\FF.$ In this case single tensors are linear forms of
bilinear forms,
$$
x\otimes y\in\!\top_{\!\X,\Y,\,\FF}\,\sse\,\Le[b[\X\x\Y,\FF],\,\FF]
=b[\X\x\Y,\,\FF]^\sharp,
$$
in the algebraic dual of $b{[\X\x\Y,\FF]}.$ Particularizing still further,
besides setting $\F=\FF$, replace the above linear space ${b[\X\x\Y,\FF]}$
by the following subset of it:
$$
b_{\X^\sharp\x\Y^\sharp}[\X\x\Y,\FF]
=\big\{{\psi\in b[\X\x\Y,\FF]}\!:{\psi(x,y)=\mu(x)\,\nu(y)}
\;\;\hbox{for}\;\;
{(\mu,\nu)\in\X^\sharp\x\Y^\sharp}\big\},
$$
consisting of products of linear forms $\mu\in\X^\sharp=\Le[\X,\,\FF]$ and
$\nu\in\Y^\sharp=\Le[\Y,\,\FF].$ The set $b_{\X^\sharp\x\Y^\sharp}[\X\x\Y,\FF]$
is not a linear manifold of $b[\X\x\Y,\,\FF].$ Define single tensors as
before$.$ To each ${(x,y)\in\X\x\Y}$ associate a function
${x\otimes y\!:b_{\X^\sharp\x\Y^\sharp}[\X\x\Y,\FF]\to\FF}$ defined for every
${\psi\in b_{\X^\sharp\x\Y^\sharp}[\X\x\Y,\FF]}$ by
$$
(x\otimes y)(\mu,\nu)=(x\otimes y)(\psi)=\psi(x,y)=\mu(x)\,\nu(y)
\quad\;\hbox{for every}\;\quad
(\mu,\nu)\in\X^\sharp\x\Y^\sharp.
$$
The difference between this and the previous procedure is due to the fact that
single tensors are not linear transformations (or linear forms) any longer as
their domain $b_{\X^\sharp\x\Y^\sharp}[\X\x\Y,\FF]$ is not a linear space$.$
However, as is readily verified they can be regarded as bilinear forms on the
Cartesian product of the linear spaces $\X^\sharp\!$ and $\Y^\sharp$:
$$
x\otimes y\in b[\X^\sharp\x\Y^\sharp,\FF]=b[\Le[\X,\FF]\x\Le[\Y,\FF],\,\FF].
$$
Thus take the collection $\!\top'_{\!\X\!,\Y}$ of all these single tensors
$$
\top'_{\!\X\!,\Y}=\big\{
x\otimes y\in b[\X^\sharp\x\Y^\sharp,\FF]
\!:\,x\in\X\;\hbox{and}\;\,y\in\Y\big\},
$$
and consider its span, $\span\!\top'_{\!\X\!,\Y}$, which is now a linear
manifold of the linear space $b[\X^\sharp\x\Y^\sharp,\FF]$ of all bilinear
forms of pairs of linear forms. As before, define a map
${\theta'\!:\X\x\Y\to\top'_{\!\X\!,\Y}\sse\span\!\top'_{\!\X\!,\Y}}$
for each pair ${(x,y)\in\X\x\Y}$ by
$$
\theta'\kern-1pt(x,y)=x\otimes y.
$$
Now the value of $\theta'$ at ${(x,y)\in\X\x\Y}$ is a bilinear form,
${\theta'\kern-1pt(x,y)\in b[\X^\sharp\x\Y^\sharp,\FF]}$, which is given by
${\theta'\kern-1pt(x,y)(\mu,\nu)=\mu(x)\,\nu(y)\in\FF}$ for every
${(\mu,\nu)\in\X^\sharp\x\Y^\sharp}$.

\begin{corollary}
\quad
${(\span\!\top'_{\!\X\!,\Y},\theta')}$ is a tensor product of\/ $\X$ and\/
$\Y$\/.
\end{corollary}

\begin{proof}
Consider the proof of Theorem 5.1$.$ Replace $\F$ by $\FF$ so that
${\Le[b[\X\x\Y,\F],\F]}$ is replaced by ${\Le[b[\X\x\Y,\FF],\,\FF]}.$ Then
replace the linear space ${b[\X\x\Y,\,\FF]}$ by the subset
$b_{\X^\sharp\x\Y^\sharp}[\X\x\Y,\,\FF]$, still keeping the same definition
of single tensors, so that $\!\top_{\!\X,\Y,\,\F}\sse{\Le[\,b[\X\x\Y,\F],\F]}$
is replaced by $\!\top'_{\!\X\!,\Y}\sse{b[\Le[\X,\FF]\x\Le[\Y,\FF],\,\FF]}.$
Again, ${\theta'\!:\!\X\x\Y\to\span\!\top'_{\!\X,\Y,\,\F}}$ is a bilinear map
with ${\span R(\theta')=\!\top'_{\!\X,\Y,\,\F}}.$ Thus the argument in the
proof of Theorem 5.1 still holds, associating with each bilinear map
${\phi\!:\X\x\Y\to\Z}$ the same linear transformation $\Phi$ into $\Z$ now
acting on $\span\!\top'_{\!\X\!,\Y}$.
\end{proof}

\begin{remark}
Here is a common and useful example of such a particular case (with
${\FF=\CC}).$ Let $\X$ and $\Y$ be complex Hilbert spaces with inner products
${\<\,\cdot\,;\cdot\,\>_\X}$ and ${\<\,\cdot\,;\cdot\,\>_\Y}$, which are
sesquilinear forms (not bilinear forms)$.$ Algebraic duals $\X^\sharp$ and
$\Y^\sharp$ are now naturally replaced by topological duals $\X^*$ and $\Y^*$
of continuous linear functionals$.$ A single tensor for each
${(x,y)\in\X\x\Y}$ is usually defined in this case by
$$
(x\otimes y)(u,v)=\<x\,;u\>_\X\,\<y\,;v\>_\Y
\quad\;\hbox{for every}\;\quad
(u,v)\in\X\x\Y.
$$
(See e.g., \cite[Section II.4]{RS} and \cite{Kub1}.) The Riesz Representation
Theorem for Hilbert spaces says that $\mu$ lies in $\X^*\!$ and $\nu$ lies in
$\Y^*\!$ if and only if ${\mu(\cdot)=\<\,\cdot\,;u\>_\X}$ and
$\nu(\cdot)=$ $\<\,\cdot\,;v\>_\Y$ for some $u$ in $\X$ and $v$ in $\Y.$ Thus
identify ${\mu\in\X^*}$ and ${\nu\in\Y^*}$ with ${u\in\X}$ and ${v\in\Y}$
such that the pair ${(u,v)\in\X\x\Y}$ is identified with the pair
${(\mu,\nu)\in\X^*\x\Y^*}\!.$ Then a single tensor ${x\otimes y}$ associated
with a pair ${(x,y)\in\X\x\Y}$ is in fact a bilinear form
${x\otimes y\!:\X^*\x\Y^*\to\CC}$ in ${b[\X^*\x\Y^*,\CC]}$, which is
equivalently written as
$$
(x\otimes y)(\mu,\nu)=\mu(x)\,\nu(y)
\quad\;\hbox{for every}\;\quad
(\mu,\nu)\in\X^*\x\Y^*.
$$
\end{remark}

\section{Final Remarks}

\subsection{Multiple Tensor Products}

It is clear how the preceding arguments (in Sec\-tions 3, 4 and 5) can be
naturally extended to cover the notion of an algebraic tensor product of a
finite collection $\{\X_i\}_{i=1}^n$ of linear spaces over the same field,
yielding a tensor product space $\bigotimes_{i=1}^n\!\X_i$ of a finite number
of linear spaces$.$ This is based on the notions of multiple Cartesian
products $\prod_{i=1}^n\!\X_i$, n-tuples, and multilinear maps as a natural
extension of Cartesian product of two linear spaces, ordered pairs, and
bilinear maps$.$ All results in Sections 3, 4 and 5 remain true (essentially
with the same statement, following similar arguments) if extended to such
multiple tensor products$.$ To extend a result on tensor product from a pair
of linear spaces to an $n$-tuple (or to an $\infty$-tuple) may be a relevant
task$.$ Sometimes this is a simple job (achieved by induction) but not
always$.$ On the other hand, what may also not be always simple is the other
way round: when a notion is initially defined for an $n$-tuple, it may be wise
to go down to a pair to see clearly what is really going on.

\subsection{Tensor Products of Banach Spaces}

We have been dealing with algebraic tensor products ${\X\otimes\Y}$ of
linear spaces $\X$ and $\Y.$ A natural follow-up is to equip the underlying
linear space ${\X\otimes\Y}$ with a norm and advance the theory of tensor
prod\-ucts to Banach spaces$.$ So a new starting point is to equip
${\X\otimes\Y}$ with a suitable norm$.$ If $\X$ and $\Y$ are Banach spaces
and $\X^*$ and $\Y^*$ are their \hbox{duals}, then let ${x\otimes y}$ and
${f\otimes g}$ be single tensors in the tensor product spaces
${\X\kern-1pt\otimes\kern-1pt\Y}$ and ${\X^*\!\otimes\kern-1pt\Y^*}\!.$
A norm $\|\cdot\|$ on ${\X\otimes\Y}$ is a {\it reasonable crossnorm}\/ if,
for every ${x\kern-1pt\in\kern-1pt\X}$, ${y\kern-1pt\in\kern-1pt\Y}$,
${f\kern-1pt\in\kern-1pt\X^*}\kern-1pt$,
${g\kern-1pt\in\kern-1pt\X^*}\kern-1pt$,
\begin{description}
\item {$\kern-12pt$\rm(a)}
$\;\|x\otimes y\|\le\|x\|\,\|y\|$,
\vskip4pt
\item {$\kern-12pt$\rm(b)}
$\;{f\otimes g}$ lies in $(\X\otimes\Y)^*\!$, $\,$and
$\|f\otimes g\|_*\le\|f\|\,\|g\|$
\quad
(where $\|\cdot\|_*$ is the norm on the dual $(\X\otimes\Y)^*$ when
$(\X\otimes\Y)$ is equipped with the norm in (a)$\kern1pt$),
\end{description}
so that ${\X^*\otimes\Y^*\sse(\X\otimes\Y)^*}\!.$ It can be verified that (i)
the above norm inequalities become identities, and (ii) when restricted to
${\X^*\otimes\Y^*}$ the norm $\|\cdot\|_*$ on ${(\X\otimes\Y)^*}$ is again a
reasonable crossnorm (with respect to ${(\X^*\otimes\Y^*)^*}).$ Two special
norms on ${\X\otimes\Y}$ are the so-called {\it injective}\/
${\|\cdot\|_{_\vee\kern-1pt}}$ and
{\it projective} $\|\cdot\|_{_\wedge\kern-1pt}$ norms,
$$
\|\digamma\|_{_\vee\kern-1pt}
=\sup_{\|f\|\le1,\,\|g\|\le1,\;f\in\X^*\!,\,g\in\Y*}
\Big|{\sum}_if(x_i)\,g(y_i)\Big|,
$$
$$
\|\digamma\|_{_\wedge\kern-1pt}
=\inf_{\{x_i\}_i,\,\{y_i\}_i,\;\digamma=\sum_ix_i\otimes y_i}
{\sum}_i\|x_i\|\,\|y_i\|,
$$
for every ${\digamma={\sum}_ix_i\otimes y_i\in\X\otimes\Y}.$ It can be shown
that (iii) these are reasonable crossnorms, and (iv) a norm $\|\cdot\|$ on
${\X\otimes\Y}$ is a reasonable crossnorm if and only if
$$
\|\digamma\|_{_\vee\kern-1pt}\le\|\digamma\|\le\|\digamma\|_{_\wedge\kern-1pt}
\quad\hbox{for every}\;\quad
\digamma\in\X\otimes\Y.
$$
Anyhow, equipped with any reasonable crossnorm, a tensor product space
${\X\otimes\Y}$ of a pair of Banach spaces $\X$ and $\Y$ is not necessarily
complete$.$ Thus one takes the completion ${\X\widehat\otimes\Y}$ of
${\X\otimes\Y}.$ For the theory of the Banach space ${\X\widehat\otimes\Y}$
the reader is referred, for instance, to \cite{Jar,DF,Rya,DFS}$.$ If $\X$ and
$\Y$ are Hilbert spaces, then ${\X\widehat\otimes\Y}$ becomes a Hilbert space
when one takes the reasonable crossnorm on ${\X\otimes\Y}$ that naturally
comes from the inner products in $\X$ and $\Y$ as in Remark 5.1, by setting
${\<x_1\otimes y_1,x_2\otimes y_2\>_{\X\otimes\Y}}
={\<x_1\otimes x_2\>_\X\,\<y_1\otimes y_2\>_\Y}$
(see, e.g., \cite{RS,Wei,Kub1}).

\section*{Acknowledgment}

The author thanks the anonymous referees for their constructive criticisms.

\bibliographystyle{amsplain}

\end{document}